\documentclass[12pt,reqno]{amsart}
\usepackage{amsmath, amsfonts, amssymb, amsthm, hyperref}
\usepackage{bm}
\allowdisplaybreaks[4]
\def\bg{\bigg}
\def\({\bg(}
\def\){\bg)}
\def\t{\text}
\def\f{\frac}

\def\eq{\equiv}

\def\Z{\mathbb Z}

\def\N{\mathbb N}

\def\1{{\bf 1}}

\def\mod #1{\ {\rm mod}\ #1}

\theoremstyle{plain}
\newtheorem{theorem}{Theorem}[section]
\newtheorem{lemma}{Lemma}
\newtheorem{corollary}{Corollary}

\theoremstyle{definition}

\newtheorem*{Acks}{Acknowledgments}
\theoremstyle{remark}
\newtheorem{remark}{Remark}

\def\<{\langle}
\def\>{\rangle}

\begin{document}
\hbox{}
\medskip

\title[Some $q$-congruences arising from certain identities]{Some $q$-congruences arising from certain identities}

\author{Chen Wang}
\address{(Chen Wang) Department of Mathematics, Nanjing
University, Nanjing 210093, People's Republic of China}
\email{cwang@smail.nju.edu.cn}

\author{He-Xia Ni*}
\address{(He-Xia Ni) Department of Applied Mathematics, Nanjing Audit University, Nanjing 211815, People's Republic of China}
\email{nihexia@yeah.net}

\subjclass[2020]{Primary 11A07, 11B65; Secondary 05A30, 05A10}
\keywords{$q$-congruences, $q$-Pochhammer symbol, cyclotomic polynomial, $q$-identities}
\thanks{*Corresponding author.}
\begin{abstract}
In this paper, by constructing some identities, we prove some $q$-analogues of some congruences. For example, for any odd integer $n>1$, we show that
\begin{gather*}
\sum_{k=0}^{n-1}\frac{(q^{-1};q^2)_k}{(q;q)_k}q^k\equiv(-1)^{(n+1)/2}q^{(n^2-1)/4}-(1+q)[n]\pmod{\Phi_n(q)^2},\\
\sum_{k=0}^{n-1}\frac{(q^3;q^2)_k}{(q;q)_k}q^k\equiv(-1)^{(n+1)/2}q^{(n^2-9)/4}+\f{1+q}{q^2}[n]\pmod{\Phi_n(q)^2},
\end{gather*}
where the $q$-Pochhanmmer symbol is defined by $(x;q)_0=1$ and $(x;q)_k=(1-x)(1-xq)\cdots(1-xq^{k-1})$ for $k\geq1$, the $q$-integer is defined by $[n]=1+q+\cdots+q^{n-1}$ and $\Phi_n(q)$ is the $n$-th cyclotomic polynomial. The $q$-congruences above confirm some recent conjectures of Gu and Guo.
\end{abstract}
\maketitle

\section{Introduction}
\setcounter{lemma}{0}
\setcounter{theorem}{0}
\setcounter{equation}{0}
\setcounter{conjecture}{0}
\setcounter{remark}{0}
\setcounter{corollary}{0}

In 2010, Sun and Tauraso \cite{ST} studied some congruence properties of sums concerning central binomial coefficients $\binom{2k}{k}$ where $k\in\N=\{0,1,\ldots\}$. For example, let $p$ be an odd prime and $r\in\Z^+$, they proved that for any $m\in \Z$ with $p\nmid m$,
\begin{align*}
\sum_{k=0}^{p^r-1}\frac{\binom{2k}{k+d}}{m^k}&\equiv u_{p^r-|d|}(m-2)\pmod{p},
\end{align*}
where $|d| \in \{0, \ldots p^r\}$ and the sequence of polynomials $u_{n}(x)$ $(n\in\N)$ is defined as follows:
$$ u_0(x)=0,\ u_1(x)=1,\  {\rm and}\ u_{n+1}(x)=xu_n(x)-u_{n-1}(x)\ (n=2,3,\ldots).$$ In particular, they obtained that
\begin{equation}\label{STcon}
\sum_{k=0}^{p^r-1}\f{\binom{2k}{k}}{2^k}\eq(-1)^{(p^r-1)/2}\pmod{p}.
\end{equation}
Later, Sun \cite{SunSciChina} further proved that \eqref{STcon} also holds modulo $p^2$.

Throughout the paper, the {\it $q$-integer} $[n]_q$ is defined as $[n]=[n]_q=1+q+\ldots+q^{n-1}$, while the {\it $q$-pochhammer symbol} ({\it $q$-shifted factorial}) is defined by $(x;q)_0=1$ and $(x;q)_k=(1-x)(1-xq)\cdots(1-xq^{k-1})$ for $k\geq1$.  And recall that the $n$-th {\it cyclotomic polynomial} $\Phi_n(q)$ is defined as
$$
\Phi_n(q)=\prod_{\substack{1\leq k\leq n\\\gcd(n,k)=1}}(q-\zeta^k),
$$
where $\zeta$ is an $n$-th primitive root of unity.

As we all know, identities or congruences usually have nice $q$-analogues. In recent years, $q$-analogues of identities and congruences have been investigated by various authors (cf. for example, \cite{AAR,Car,GuGuo,GZ2,GZ3,GZ4,Guoijnt,GuoZeng2010,NP,R1,Ta}). In 2010, Guo and Zeng \cite{GuoZeng2010} gave the following $q$-analogue of \eqref{STcon}:
\begin{equation}\label{GZcon}
\sum_{k=0}^{n-1}\f{(q;q^2)_k}{(q;q)_k}q^k\eq(-1)^{(n-1)/2}q^{(n^2-1)/4}\pmod{\Phi_n(q)},
\end{equation}
where $n$ is a positive odd integer. Morever, Guo \cite{Guoijnt} established the following generalization of \eqref{GZcon}:
\begin{equation}\label{guokey}
\sum_{k=0}^{n-1}\f{(q;q^2)_k}{(q;q)_k}q^k\eq(-1)^{(n-1)/2}q^{(n^2-1)/4}\pmod{\Phi_n(q)^2}.
\end{equation}
 It should be pointed out that \eqref{guokey} for odd primes $n$ was first conjectured by Tauraso \cite{Ta} in 2013.

Recently, Gu and Guo \cite{GuGuo} provided some $q$-congruences formally analogous to \eqref{GZcon} by making use of Carlitz's transformation formula (cf. \cite{Car}). For any odd integer $n>1$, they proved that
\begin{gather}
\label{GGcon1}\sum_{k=0}^{n-1}\f{(q^{-1};q^2)_k}{(q;q)_k}q^k\eq(-1)^{(n+1)/2}q^{(n^2-1)/4}\pmod{\Phi_n(q)},\\
\label{GGcon2}\sum_{k=0}^{n-1}\f{(q^3;q^2)_k}{(q;q)_k}q^k\eq(-1)^{(n+1)/2}q^{-(n-3)^2/4}\pmod{\Phi_n(q)}.
\end{gather}

Our first theorem concerns the generalization of \eqref{GGcon1}.
\begin{theorem}\label{maintheorem1}
For any odd integer $n>1$ we have
\begin{gather}
\label{mainth1id1}\sum_{k=0}^{n-1}\f{(q^{-1};q^2)_k}{(q;q)_k}q^k\eq(-1)^{(n+1)/2}q^{(n^2-1)/4}-(1+q)[n]\pmod{\Phi_n(q)^2},\\
\label{mainth1id2}\sum_{k=0}^{n-1}\f{(q^{-1};q^2)_k}{(q;q)_k}q^{2k}\eq(-1)^{(n+1)/2}q^{(n^2+3)/4}-2q[n]\pmod{\Phi_n(q)^2}.
\end{gather}
\end{theorem}

Letting $q\to1$ in \eqref{mainth1id1} or \eqref{mainth1id2} we obtain the following congruence.
\begin{corollary}\label{coro1}
Let $p$ be an odd prime and $r\in\Z^+$. Then
\begin{equation}\label{coro1id}
\sum_{k=0}^{p^r-1}\f{1}{2^k(2k-1)}\binom{2k}{k}\eq(-1)^{(p-1)/2}+2p^r\pmod{p^2}.
\end{equation}
\end{corollary}
\begin{remark}
\eqref{mainth1id1} and \eqref{coro1id} were conjectured by Gu and Guo in \cite{GuGuo}. \eqref{mainth1id2} is actually a different $q$-analogue of \eqref{coro1id}.
\end{remark}

Gu and Guo \cite{GuGuo} also attempted to find a$\mod \Phi_n(q)^2$ extension of \eqref{GGcon2} but failed. The next theorem gives a different $q$-analogue of \eqref{STcon} and generalizes \eqref{GGcon2}.
\begin{theorem}\label{maintheorem2}
Let $n>1$ be an odd integer. Then
\begin{gather}
\label{mainth2id1}\sum_{k=0}^{n-1}\f{(q;q^2)_k}{(q;q)_k}q^{2k}\eq(-1)^{(n-1)/2}q^{(n^2-5)/4}+\f{q-1}{q}[n]\pmod{\Phi_n(q)^2},\\
\label{mainth2id2}\sum_{k=0}^{n-1}\f{(q^3;q^2)_k}{(q;q)_k}q^k\eq(-1)^{(n+1)/2}q^{(n^2-9)/4}+\f{1+q}{q^2}[n]\pmod{\Phi_n(q)^2}.
\end{gather}
\end{theorem}

Letting $q\to1$ we have the following corollary which confirms \cite[(1.8)]{GuGuo}.
\begin{corollary}\label{coro2}
Let $p$ be an odd prime and $r\in\Z^+$. Then
\begin{equation}\label{coro2id}
\sum_{k=0}^{p^r-1}\f{(2k+1)}{2^k}\binom{2k}{k}\eq(-1)^{(p+1)/2}+2p^r\pmod{p^2}.
\end{equation}
\end{corollary}

\begin{remark}
\eqref{mainth2id2} is essentially an extension of \eqref{GGcon2}. In fact, noting that $q^n\eq1\pmod{\Phi_n(q)}$, we immediately get
$$
q^{(n^2-9)/4}\eq q^{-(n-3)^2/4}\pmod{\Phi_n(q)}.
$$
\end{remark}

Differently from Gu and Guo's method, we will not use Carlitz's transformation. Our strategy is to find some new identities linking the $q$-congruences to be solved with \eqref{guokey}. Assume that $F(k,q)$ is a rational function in $q$ such that $F(k,q)/F(k-1,q)$ can be written as a ratio of two polynomials in $q$. Now we want to find a polynomial $R(k,q)$ such that $\sum_{k=0}^nF(k,q)R(k,q)$ has a closed form. Consider the summation $\sum_{k=0}^nF(k,q)$. By the definition of $F$, we may write
$$
\f{F(k,q)}{F(k-1,q)}=\f{S(k,q)}{T(k,q)},
$$
or equivalently,
\begin{equation}\label{introid1}
F(k,q)T(k,q)=F(k-1,q)S(k,q),
\end{equation}
where $S(k,q)$ and $T(k,q)$ are polynomials of $q$. Then summing both sides of \eqref{introid1} from $k=1$ to $n$ and via some simple computation we find that
\begin{equation}\label{introid2}
\sum_{k=0}^n(T(k,q)-S(k+1,q))F(k,q)=F(0,q)T(0,q)-F(n,q)S(n+1,q).
\end{equation}
Here $T(k,q)-S(k+1,q)$ is the polynomial that we hope to find.

The proofs of Theorem \ref{maintheorem1} and \ref{maintheorem2} will be given in Sections 2 and 3 respectively.

\section{Proof of Theorem \ref{maintheorem1}}
\setcounter{lemma}{0}
\setcounter{theorem}{0}
\setcounter{equation}{0}
\setcounter{conjecture}{0}
\setcounter{remark}{0}
\begin{lemma}\label{th1lem1}
For any positive integer $n$, we have the following identities
\begin{equation}\label{th1lem1id1}
\sum_{k=0}^{n-1}\f{(q;q^2)_k}{(q;q)_k}q^k+\f{1}{1-q}\sum_{k=0}^{n-1}\f{(q^{-1};q^2)_k}{(q;q)_k}q^k(1-q^k)=\f{(q;q^2)_{n-1}}{(q;q)_{n-1}}q^{n-1}
\end{equation}
and
\begin{equation}\label{th1lem1id2}
\f{1}{1-q}\sum_{k=0}^{n-1}\f{(q^{-1};q^2)_k}{(q;q)_k}q^k(q-q^k)=-\f{(q;q^2)_{n-1}}{(q;q)_{n-1}}.
\end{equation}
\end{lemma}

\begin{proof}
Set
$$
F_1(k,q)=\f{(q^{-1};q^2)_k}{(q;q)_k}q^k.
$$
It is easy to verify that
\begin{equation}\label{th1lem1pfid1}
(1-q^k)F_1(k,q)=q(1-q^{2k-3})F_1(k-1,q).
\end{equation}
Summing both sides of \eqref{th1lem1pfid1} from $k=1$ to $n-1$ and noting that the left-hand side of \eqref{th1lem1pfid1} vanishes when $k=0$, we arrive at
$$
\sum_{k=0}^{n-1}(1-q^k)F_1(k,q)=\sum_{k=1}^{n-1}q(1-q^{2k-3})F_1(k-1,q)=\sum_{k=0}^{n-2}q(1-q^{2k-1})F_1(k,q),
$$
or equivalently,
$$
\sum_{k=0}^{n-1}q(1-q^{2k-1})F_1(k,q)-\sum_{k=0}^{n-1}(1-q^k)F_1(k,q)=q(1-q^{2n-3})F_1(n-1,q).
$$
Then \eqref{th1lem1id1} follows by noting that
$$
q(1-q^{2k-1})F_1(k,q)=(q-1)\f{(q;q^2)_k}{(q;q)_k}q^k
$$
for all $k$ among $0,1,\ldots,n-1$.

To show \eqref{th1lem1id2} we set
$$
F_2(k,q)=\f{(q^{-1};q^2)_k}{(q;q)_k}.
$$
Now we find that
$$
(1-q^k)F_2(k,q)=(1-q^{2k-3})F_2(k-1,q).
$$
Then we may obtain \eqref{th1lem1id2} by some similar arguments as above.
\end{proof}

\begin{lemma}\label{th1lem2}
For any odd integer $n>1$, we have
\begin{equation}\label{th1lem2id}
\f{(q;q^2)_{n-1}}{(q;q)_{n-1}}\eq-q[n]\pmod{\Phi_n(q)^2}.
\end{equation}
\end{lemma}

\begin{proof}
Clearly,
\begin{equation}\label{th1lem2pfid1}
\f{(q;q^2)_{n-1}}{(q;q)_{n-1}}=\f{(q;q)_{2n-2}}{(-q;q)_{n-1}(q;q)_{n-1}^2}.
\end{equation}
Note that
$$
q^n\eq1\pmod{\Phi_n(q)}
$$
and
$$
q^j\not\eq1\pmod{\Phi_n(q)}\quad \t{for all}\ j=1,2,\ldots,n-1.
$$
Thus we have
\begin{equation}\label{th1lem2pfid2}
\f{(q;q)_{2n-2}}{(q;q)_{n-1}^2}=(1-q^n)\f{\prod_{j=1}^{n-2}(1-q^{n+j})}{\prod_{j=1}^{n-1}(1-q^j)}\eq\f{1-q^n}{1-q^{-1}}=-q[n]\pmod{\Phi_n(q)^2}.
\end{equation}
By \cite[Corollary 10.2.2(c)]{AAR} we have
\begin{equation}\label{th1lem2pfid3}
(-q;q)_{n-1}=\f{(-q;q)_n}{1+q^n}\eq\f12\sum_{k=0}^n\f{(q;q)_nq^{n(n+1)/2}}{(q;q)_k(q;q)_{n-k}}\eq1\pmod{\Phi_n(q)}.
\end{equation}
Substituting \eqref{th1lem2pfid2} and \eqref{th1lem2pfid3} into \eqref{th1lem2pfid1} we immediately obtain the desired lemma.
\end{proof}

\medskip
\noindent{\it Proof of Theorem \ref{maintheorem1}}. We first prove \eqref{mainth1id1}. Combining \eqref{th1lem1id1} and \eqref{th1lem1id2} and with the help of Lemma \ref{th1lem2} we obtain that
\begin{align*}
\sum_{k=0}^{n-1}\f{(q;q^2)_k}{(q;q)_k}q^k+\sum_{k=0}^{n-1}\f{(q^{-1};q^2)_k}{(q;q)_k}q^k=&\f{(q;q^2)_{n-1}}{(q;q)_{n-1}}(1+q^{n-1})\notag\\
\eq&-q[n](1+q^{-1})\notag\\
=&-(1+q)[n]\pmod{\Phi_n(q)^2}.
\end{align*}
Now \eqref{mainth1id1} follows from \eqref{guokey}.
\medskip

With the help of \eqref{th1lem1id2}, we have
$$
q\sum_{k=0}^{n-1}\f{(q^{-1};q^2)_k}{(q;q)_k}q^k-\sum_{k=0}^{n-1}\f{(q^{-1};q^2)_k}{(q;q)_k}q^{2k}=(q-1)\f{(q;q^2)_{n-1}}{(q;q)_{n-1}}.
$$
Then we obtain \eqref{mainth1id2} by noting \eqref{mainth1id1} and Lemma \ref{th1lem2}.

The proof of Theorem \ref{maintheorem1} is now complete.\qed

\section{Proof of Theorem \ref{maintheorem2}}
\setcounter{lemma}{0}
\setcounter{theorem}{0}
\setcounter{equation}{0}
\setcounter{conjecture}{0}
\setcounter{remark}{0}

\begin{lemma}\label{th2lem1}
For any positive integer $n$ we have the following identities.
\begin{equation}\label{th2lem1id1}
(1-q)\sum_{k=0}^{n-1}\f{(q^3;q^2)_k}{(q;q)_k}q^k-\f{1}{q}\sum_{k=0}^{n-1}\f{(q;q^2)_k}{(q;q)_k}q^k(1-q^k)=\f{(q^3;q^2)_{n-1}}{(q;q)_{n-1}}(q^{n-1}-q^n)
\end{equation}
and
\begin{equation}\label{th2lem1id2}
\sum_{k=0}^{n-1}\f{(q;q^2)_k}{(q;q)_k}q^k(1-q^{k+1})=(1-q)\f{(q^3;q^2)_{n-1}}{(q;q)_{n-1}}.
\end{equation}
\end{lemma}

\begin{proof}
Set
$$
G_1(k,q)=\f{(q;q^2)_k}{(q;q)_k}q^k.
$$
Then we may easily check that
$$
(1-q^k)G_1(k,q)=(q-q^{2k})G_1(k-1,q).
$$
Summing both sides of the above identity from $k=1$ to $n-1$ we have
$$
q\sum_{k=0}^{n-1}(1-q^{2k+1})G_1(k,q)-\sum_{k=0}^{n-1}(1-q^k)G_1(k,q)=(q-q^{2n})G_1(n-1,q).
$$
Thus we obtain \eqref{th2lem1id1} by noting
$$
(1-q)(q^3;q^2)_k=(1-q^{2k+1})(q;q^2)_k
$$
for all $k$ among $0,1,\dots,n-1$.

Also, \eqref{th2lem1id2} can be deduced by setting
$$
G_2(k,q)=\f{(q;q^2)_k}{(q;q)_k}
$$
and noting that
$$
(1-q^k)G_2(k,q)=(1-q^{2k-1})G_2(k-1,q).
$$
\end{proof}
\medskip

\noindent{\it Proof of Theorem \ref{maintheorem2}}. By \eqref{th2lem1id2} we have
$$
\sum_{k=0}^{n-1}\f{(q;q^2)_k}{(q;q)_k}q^k-q\sum_{k=0}^{n-1}\f{(q;q^2)_k}{(q;q)_k}q^{2k}=(1-q^{2n-1})\f{(q;q^2)_{n-1}}{(q;q)_{n-1}}.
$$
Then \eqref{mainth2id1} follows from \eqref{guokey} and Lemma \ref{th1lem2}. Substituting \eqref{guokey} and \eqref{mainth2id1} we immediately get \eqref{mainth2id2}.

The proof of Theorem \ref{maintheorem2} is now complete.\qed

\begin{Acks}
The first author is supported by the National Natural Science Foundation of China (Grant No. 11971222).
\end{Acks}

\end{document}